\documentclass[12pt]{amsart}
\usepackage{amssymb,amsmath, cite}
\makeatletter
\def\@cite#1#2{{\m@th\upshape\bfseries%
[{#1\if@tempswa{\m@th\upshape\mdseries, #2}\fi}]}}
\makeatother
\theoremstyle{plain}
\newtheorem{thm}{Theorem}[section]

\newtheorem{prop}[thm]{Proposition}
\newtheorem{lem}[thm]{Lemma}

\theoremstyle{definition}

\newcommand{\bC}{{\mathbb{C}}}
\newcommand{\bR}{{\mathbb{R}}}
\newcommand{\bZ}{{\mathbb{Z}}}
\newcommand{\W}{{\mathcal{W}}}
\renewcommand{\S}{{\mathcal{S}}}
\newcommand{\ad}{\operatorname{ad}}
\newcommand{\Ad}{\operatorname{Ad}}

\newcommand{\im}{\operatorname{Im}}
\newcommand{\rank}{\operatorname{rank}}

\newcommand{\spa}{\operatorname{span}}

\newcommand{\corank}{\operatorname{corank}}

\newcommand{\Fg}{{\mathfrak{g}}}

\newcommand{\Ft}{{\mathfrak{t}}}
\newcommand{\Fn}{{\mathfrak{n}}}
\newcommand{\su}{{\mathfrak{su}}}
\newcommand{\so}{{\mathfrak{so}}}
\renewcommand{\sp}{{\mathfrak{sp}}}

\newcommand{\id}{\operatorname{Id}}
\begin{document}


\title[Sums of adjoint orbits]{Sums of adjoint orbits \\ and $L^2$--singular dichotomy for $SU(m)$ }
%
\author[A.Wright]{Alex~Wright\\The University of Chicago}
\address{Math.\ Dept.\\U. Chicago\\
5734 S. University Avenue\\
Chicago, Illinois 60637}
\email{amwright@math.uchicago.edu}
%
\date{}


\begin{abstract}
Let $G$ be a real compact connected simple Lie group, and $\Fg$ its Lie algebra. We study the problem of determining, from root data, when a sum of adjoint orbits in $\Fg$, or a product of conjugacy classes in $G$, contains an open set.

Our general methods allow us to determine exactly which sums of adjoint orbits in $\su(m)$ and products of conjugacy classes in $SU(m)$ contain an open set, in terms of the highest multiplicities of eigenvalues. 

For $\su(m)$ and $SU(m)$ we show $L^2$--singular dichotomy: The convolution of invariant measures on adjoint orbits, or conjugacy classes, is either singular to Haar measure or in $L^2$. 

\end{abstract}

\maketitle


\section{Introduction}\label{S:intro}

We will consider a real compact connected simple Lie group $G$ with Lie algebra $\Fg$. The adjoint orbit of $X\in \Fg$ is defined to be $O_X=\{\Ad(g)X: g\in G\}$, and the conjugacy class of $x\in G$ is defined to be $C_x=\{gxg^{-1}: g\in G\}$. We wish to determine when a sum of adjoint orbits, or product of conjugacy classes, contains an open set. In this paper we give a sufficient condition, and a necessary condition, for this to occur. For the special unitary group, these two results allow a complete solution to the problem.

\begin{thm}\label{T:suopens}
Suppose $X_1,\ldots, X_k\in \su(m)$ (or $x_1,\ldots, x_k\in SU(m)$) and that $X_i$'s (respectively $x_i$'s) highest eigenvalue multiplicity is $q_i$. Then the sum of adjoint orbits $\sum_{i=1}^k O_{X_i}$ (respectively the product of conjugacy classes $\prod_{i=1}^k C_{x_i}$) contains an open set if and only if $$\sum_{i=1}^k q_i/m\leq k-1$$ and the tuple of $X_i$ (respectively $x_i$) does not consist of exactly two matrices each with exactly two eigenvalues of equal multiplicity at least two. 
\end{thm}

Each adjoint orbit and conjugacy class supports a natural $G$--invariant probability measure. The convolution of such measures is supported on a sum of adjoint orbits or a product of conjugacy classes, and is absolutely continuous to Haar measure if and only if this sum, or product, contains an open set \cite{Ra2}. 

After having proven Theorem \ref{T:suopens}, with only very little extra effort, we get the following \emph{$L^2$--singular dichotomy}. Note that since the measures are compactly supported, containment in $L^2$ implies containment in $L^1$. As can be seen for the case of two copies of the same orbit in $\su(2)$, the Radon-Nikodym derivative of the convolution need not be bounded. 

\begin{thm}\label{T:dich}
Let $\mu_1,\ldots, \mu_k$ be invariant measures on adjoint orbits in $\su(m)$. Then either $\mu_1 *\cdots *\mu_k$ is singular to Haar measure on $\su(m)$ or $\mu_1*\cdots *\mu_k\in L^2(\su(m))$. The same statement holds for invariant measures on conjugacy classes in $SU(m)$. 
\end{thm}

In the case of $\su(m)$ and $SU(m)$ this generalizes a result of Gupta, Hare and Seyfaddini \cite{GH3, GHS}, who have proven $L^2$--singular dichotomy for convolution powers of a single orbital measure in a classical compact Lie group. In so doing, again in the four classical casses, they have found the minimal $k$ (depending on the orbit in question) such that a sum of $k$ copies of a single adjoint orbit, or a product of $k$ copies of a single conjugacy class, contains an open set. For consistency with this literature, we speak of adjoint orbits instead of co-adjoint orbits. As the group is compact, there is no difference.  

The following facts will not be used, but provide insight and context. There is a relationship between sums of adjoint orbits and products of conjugacy classes: $\prod_{i=1}^k C_{\exp(X_i)}\subset \exp(\sum_{i=1}^k O_{X_i})$, with equality if all $X_i$ are sufficiently small \cite{DW}. A sum of adjoint orbits is the image of a moment map, so it follows by convexity properties of moment maps of Guillemin-Sternberg and Kirwan \cite{GS1, GS2, K} that the intersection of a sum of adjoint orbits with a positive Weyl chamber of a maximal torus is a convex polytope. In the case of $\su(m)$, a sum of two adjoint orbits can be described by a system of explicit linear inequalities on the eigenvalues, as recently proven by Knutson and Tao in their solution to Horn's problem \cite{KT}. Again in the case of $\su(m)$, the Radon-Nikodym derivative of the convolution of two invariant measures on adjoint orbits has been computed \cite{FG}. There is a related general formula for the convolution of two invariant measures on adjoint orbits in terms of the projection of such measures to maximal tori \cite{DRW}. Up to scaling, the convolution of invariant measures on adjoint orbits is the push-forward of Liouville measure under the moment map, and is hence a Duistermaat-Heckman measure. In this context, the Radon-Nikodym derivative might be studied by generalizing the Guillemin-Lerman-Sternberg formula to non-abelian groups. A conjugacy class in $G$ is not a symplectic manifold, but nonetheless the product of conjugacy classes is the image of a Lie group valued moment map \cite{AMM}. 

Our general results are based in root theory. The root space decomposition of the complexification $\Fg_\bC$ gives that, for some root system $\Phi$, $$\Fg_\bC=\Ft_\bC\oplus\bigoplus_{\alpha\in\Phi}\Fg_\alpha.$$ If $\Fg$ is simple, then the root system $\Phi$ is irreducible, and hence must be of classical type $A_n, B_n, C_n$ or $D_n$ or of exceptional type $E_6, E_7, E_8, F_4$ or $G_2$. Types $A_n, B_n, C_n$ and $D_n$ correspond to the real compact Lie algebras $\su(n+1), \so(2n+1), \sp(n, \bC)$ and $\so(2n)$ respectively.

Every $X\in\Fg$ lies in some maximal torus $\Ft$. Our convention will be that roots are real-valued on $\Ft$ and defined modulo $2\pi \bZ$ on the corresponding maximal torus $T$ of $G$. We define the \textit{annihilating root subsystem} of $X$ as $\Phi_X=\{\alpha\in\Phi:\alpha(X)=0\}$. We similarly define $\Phi_x=\{\alpha\in\Phi:\alpha(X)\in 2\pi \bZ \}$ for $x\in T$. The annihilating root subsystems may depend on the choice of maximal torus, but are well defined up to the action of the Weyl group $\W$. 

We say that a root subsystem $\Psi$ of $\Phi$  is $\bR$--closed if $\Psi=\spa_{\bR}(\Psi)\cap\Phi$. When considering an adjoint orbit or conjugacy class, we will only use the information of the annihilating root system, and there will not be much difference between the group (conjugacy class) case and the algebra (adjoint orbit) case. However, in general, if $X\in \Fg$ and $x\in G$, then $\Phi_X$ is $\bR$--closed, but $\Phi_x$ need not be, so there are more possible annihilating subsystems for the group. All root subsystems of a root system of type $A_n$ are $\bR$--closed, so for the special unitary group the group and algebra problems are essential identical.   

We define the rank of $\Psi$ as $\rank(\Psi)=\dim(\spa_{\bR}\Psi)$, and the co-rank as $\corank(\Psi)=\rank(\Phi)-\rank(\Psi)$. For convenience we define $N_\Psi=\Phi\setminus\Psi$ and $N_{X}=\Phi\setminus\Phi_X$. 

To prove a sum does not contain an open set we use a transversality argument.  To prove a sum of adjoint orbits contains an open set we compute the rank of certain map, which essentially allows us to estimate the dimension of the set of critical points of the addition map $\prod_{i=1}^k O_{X_i}\to \Fg$. In terms of symplectic geometry (which we will not rely on), our starting point for everything is that the image of the moment map contains an open set if and only if the principal stabilizer of the diagonal action is finite. This corresponds to our Theorem \ref{P:main}, for which we give an elementary proof. The following are our two general results. Generalizations are likely possible to the moment map for the diagonal action on a product of symplectic manifolds with Hamiltonian $G$ actions, each of whose principal stabilizers is explicitly known. Since such generalizations are not relevant to our task, we will not comment further.   
 
\begin{thm}\label{T:open}
Suppose $X_1, \ldots X_k\in \Fg$, and write $N_i=N_{X_i}$ (or suppose $x_1,\ldots, x_k\in G$,  and $N_i=N_{x_i}$). Suppose
 $$\sum_{i=1}^k \min_{\sigma\in\W}|N_{i}\cap\sigma N_\Psi|\geq|N_\Psi| +1$$ for all $\bR$--closed co-rank one root subsystems $\Psi$ of $\Phi$. Then $\sum_{i=1}^k O_{X_i}$ (respectively $\prod_{i=1}^k C_{x_i}$) contains an open set.
 \end{thm}

\begin{thm}\label{T:sing}
Suppose $X_1, \ldots X_k\in \Fg$, and write $N_i=N_{X_i}$ (or suppose $x_1,\ldots, x_k\in G$,  and $N_i=N_{x_i}$). Suppose that $\sigma_i(N_{i}) \cap N_{\Psi}$, $i=1, \ldots, k$, are disjoint for some $\bR$--closed co-rank one root subsystem $\Psi$ of $\Phi$ and some $\sigma_1, \ldots, \sigma_k\in \W$. Then $\sum_{i=1}^k O_{X_i}$ (respectively $\prod_{i=1}^k C_{x_i}$) does not contain an open set. 
\end{thm}
 
We will use the following theorem of Gupta, Hare, and Seyfaddini \cite{GHS} only to prove the $L^2$--singular dichotomy. The group case is proven by estimating the size of characters of $G$ by using l'H\^{o}pital's rule to evaluate the Weyl character formula, and applying this estimate to the operator-valued Fourier transform of a convolution of measures. Then the result can be transferred from the group to the algebra using the wrapping map of Dooley and Wildberger. 
 
\begin{thm}\label{T:L2}
Suppose $X_1, \ldots X_k\in \Fg$, and write $N_i=N_{X_i}$ (or suppose $x_1,\ldots, x_k\in G$,  and $N_i=N_{x_i}$). Let $\mu_i$ be the invariant probability measure supported on $O_{X_i}$ (respectively $C_{x_i}$). Suppose that
 $$\sum_{i=1}^k \max_{\sigma\in\W}|N_i\cap\sigma N_\Psi|\geq |N_\Psi|+\corank(\Psi)$$ for all proper $\bR$--closed root subsystems $\Psi$ of $\Phi$. Then $\mu_1*\cdots *\mu_k\in L^2$.
 \end{thm}

The organization of the paper is as follows. In Section \ref{S:mains} we prove our two general results, Theorems \ref{T:open} and \ref{T:sing}. In Section \ref{S:su_opens} we use these theorems to determine which sums of adjoint orbits in $\su(m)$ contain an open set, giving Theorem {T:suopens}. In Section \ref{S:su_di} we prove our $L^2$--singular dichotomy, Theorem \ref{T:dich}.

It would be desirable to have analogues of the $SU(m)$ results for the other classical compact Lie algebras and groups. Our singularity result, Theorem \ref{T:sing}, does not seem to be strong enough for this task. The author is not aware of any case where Theorem \ref{T:open} is not sharp. 

\emph{Acknowledgements.} This research was supported in part by NSERC. The author thanks Kathryn Hare for many useful conversations and much encouragement, Yael Karshon for sharing her expertise on symplectic geometry, and Matthew Strom Borman for giving helpful comments on a draft of this paper.


\section{Combinatorial root conditions}\label{S:mains}

Throughout this paper, $G$ will denote a real compact connected simple Lie group with Lie algebra $\Fg$ and root system $\Phi$. We will fix a maximal torus $\Ft$ throughout the paper and assume that $X_1,\ldots, X_k\in \Ft$, and $x_1, \ldots, x_k$ are contained in $T=\exp(\Ft)$.

The tangent space to $O_{X}$ at $\Ad(g)X$ is $\im\ad(\Ad(g)X)$.  Since $G$ is compact, we can endow $\Fg$ with an $\Ad$--invariant inner product $(\cdot,\cdot)$. With respect to this inner product, $\ad(X)$ is skew symmetric, so we have $\im\ad(X)=(\ker\ad(X))^\perp$. We define $\Fn_X=\ker\ad(X)$ to be the null space of $\ad(X)$.

Since $G$ is compact, it is also a real linear algebraic group. By subvarieties of $G$, we will mean affine algebraic subvarieties. In the argument below, it suffices to know that these are in particular real analytic subvarieties.

Our starting point for studying sums of adjoint orbits is the following proposition:

\begin{prop}\label{P:main}
The sum  $\sum_{i=1}^k O_{X_i}$ contains an open set if and only if there exist $(g_1,\ldots,g_k)\in G^k$ so that $$\bigcap_{i=1}^k \Ad(g_i)\Fn_{X_i}=\{0\}.$$ Furthermore, if $\sum_{i=1}^k O_{X_i}$ contains an open set, then this intersection is $\{0\}$ for all $(g_1,\ldots,g_k)$ off a proper subvariety of $G^k$. 
\end{prop}

\begin{proof}
We consider the addition map from $O_{X_1}\times\cdots\times O_{X_k}$ to $\Fg$. The image of its derivative at $(\Ad(g_1)X_1,\ldots,\Ad(g_k)X_k)$ is
\begin{eqnarray*}
\sum_{i=1}^k \im\ad(\Ad(g_i)X_i)
&=&
\sum_{i=1}^k \Ad(g_i)\im\ad(X_i)
\\&=&
\sum_{i=1}^k \Ad(g_i)\Fn_{X_i}^\perp
\\&=&
\left(\bigcap_{i=1}^k \Ad(g_i)\Fn_{X_i}\right)^\perp.
\end{eqnarray*}
Hence the derivative of the addition map has full rank at this point if and only if $\bigcap_{i=1}^k \Ad(g_i)\Fn_{X_i}=\{0\}$. The image of the addition map is $\sum_{i=1}^k O_k$. By Sard's theorem and the open mapping theorem it contains an open set if and only if the derivative of the addition map has full rank at some point.

If the addition map has full rank at some point, then some $\dim \Fg$ by $\dim \Fg$ minor of its derivative has non-zero determinant.  The zero set of this determinant is a proper subvariety of $G^k$.
\end{proof}

A very similar theorem holds for products of conjugacy classes. We define $\Fn_{x_i}=\ker (\Ad(x_i^{-1})-\id)$, and note that the tangent space to $C_{x_i}$ at $g_i$ is $(\Ad(g_i^{-1}) \Fn_{x_i})^\perp$. Note that $\Ad(g)\Fn_x=\Fn_{gxg^{-1}}$. 

\begin{prop}\label{P:main_conf}
The product  $\prod_{i=1}^k C_{x_i}$ contains an open set if and only if there exist $(g_1,\ldots,g_k)\in G^k$ so that $$\bigcap_{i=1}^k \Ad(g_i)\Fn_{x_i}=\{0\}.$$ Furthermore if $\prod_{i=1}^k C_{x_i}$ contains an open set then this intersection is $\{0\}$ for all $(g_1,\ldots,g_k)$ off a proper subvariety of $G^k$. 
\end{prop}

The proof of this proposition is left to the reader, who should be aware that the above intersection is the orthogonal complement of the image of the derivative (of the natural product map) at a point which is not $(g_1, \ldots, g_k)$.  

We wish to understand when $\cap_{i=1}^k \Ad(g_i)\Fn_{X_i}$, or $\cap_{i=1}^k \Ad(g_i)\Fn_{x_i}$, is trivial. The following lemma says that this happens exactly when the intersection contains a certain type of element, called \emph{maximally singular}. So it will suffice to study when this intersection contains this type of element. 

To motivate the next lemma, we mention that in $\su(m)$, the maximally singular elements will be those matrices with exactly two distinct eigenvalues. In $\su(m)$, if $g_1X_1g_1^{-1},\ldots, g_kX_kg_k^{-1}$ all commute with some $Z\in \su(m)$ (meaning $Z\in \cap_{i=1}^k \Ad(g_i)\Fn_{X_i}$), then $Z$ may be taken to have only two distinct eigenvalues. This can be proven using the fact that diagonalizable matrices commute if and only if they preserve each other's eigenspaces. 

We return now to the general situation. For a root subsystem $\Phi_0$ of $\Phi$, define $\Ft_{\Phi_0}=\{T\in\Ft|\alpha(T)=0, \forall\alpha\in\Phi_0\}$. For $Z\in \Ft$, define $\Ft_Z$ to be $\Ft_{\Phi_Z}$.  

For $X\in \Fg$ define $G_X=\{g\in G: \Ad(g)X=X\}$, and for $x\in G$ define $G_x=\{g\in G: gxg^{-1}=x\}$. Note that $G_X$ and $G_x$ have Lie algebras $\Fn_X$ and $\Fn_x$ respectively. Note also that $\Ft_X$ is the center of $\Fn_X$ and so $\Ft_X$ is contained in any maximal torus that contains $X$. 

We will call $Z\in \Fg$ \emph{maximally singular} if $\Phi_Z$ has co-rank one. 

\begin{lem}\label{L:maxbad}
Given $g_1,\ldots, g_k\in G$, if  the intersection $\cap_{i=1}^k \Ad(g_i)\Fn_{X_i}$ or $\cap_{i=1}^k \Ad(g_i)\Fn_{x_i}$ is not $\{0\}$, then it contains a maximally singular element. 
\end{lem}

\begin{proof}
For the algebra case, take $Z \in \cap_{i=1}^k \Ad(g_i)\Fn_{X_i}$. Since $Z$ and $\Ad(g_i)X_i$ are simultaneously contained in a maximal torus, we see that $\Ft_{Z}\subset  \Ad(g_i)\Fn_{X_i}$ for all $i$. So it suffices to pick a maximally singular element in $\Ft_{Z}$. To do so, take any $\bR$--closed co-rank one system $\Psi$ containing $\Phi_Z$, and pick any non-zero element of $\Ft_\Psi\subset \Ft_Z$.

The group case is very similar. Given  $Z \in \cap_{i=1}^k \Ad(g_i)\Fn_{x_i}$, we have that $\exp(\bR\cdot Z)$ and $g_i x_i g_i^{-1}$ commute. Since maximal tori are in fact maximal abelian subgroups, we see that $\exp(\bR\cdot Z)$ and $g_i x_i g_i^{-1}$ are simultaneously contained in a maximal torus of $G$. Hence $\Ft_{Z}\subset\Ad(g_i)\Fn_{x_i}$ and we proceed as above. 
\end{proof}

There are only finitely many co-rank one $\bR$--closed root subsystems $\Psi$ of $\Phi$, so we can pick a finite set $\S$ of $Z\in \Ft$ such that $\Phi_Z$ runs over all co-rank one $\bR$--closed root subsystems. For $Z\in \S$, for the algebra case, the lemma motivates us to study the zeros of the function $f_Z: O_{Z}\times G^k\to \Fg^k$ defined by $$f_Z(Z', g_1,\ldots, g_k)=([Z', \Ad(g_1) X_1], \ldots, [Z', \Ad(g_k) X_k]).$$

For the group case we are motivated to study a similar function $g_Z:O_{Z}\times G^k\to \Fg^k$ defined by $$g_Z(Z', g_1,\ldots, g_k)=((\Ad(g_1x_1^{-1}g_1^{-1})-\id)Z', \ldots, (\Ad(g_kx_k^{-1}g_k^{-1})-\id)Z').$$

Note that $f_Z(Z', g_1, \ldots, g_k)=0$ if and only if $Z'\in \cap_{i=1}^k \Ad(g_i)\Fn_{X_i}$ and $g_Z(Z', g_1, \ldots, g_k)=0$ if and only if $Z'\in \cap_{i=1}^k \Ad(g_i)\Fn_{x_i}$.

In order to show a sum of adjoint orbits contains an open set, we will argue, in the proof of Theorem \ref{T:open}, that the zero set of $f_Z$ has high co-dimension (for each possible type of $Z$), giving that $\cap_{i=1}^k \Ad(g_i)\Fn_{X_i}$ cannot always contain a conjugate of $Z$. To control the co-dimension of the zero set of $f_Z$, we bound the rank of $f_Z$. 

\begin{lem}\label{L:rank}
The rank of $Df_Z$ at $(Z', g_1,\ldots, g_k)\in f_Z^{-1}(0)$, or of  $Dg_Z$ at $(Z', g_1,\ldots, g_k)\in g_Z^{-1}(0)$, is at least $$\sum_{i=1}^k \min_{\sigma\in \W} |N_{\Psi}\cap\sigma N_{i}|.$$ 
\end{lem}

\begin{proof}
We begin with the algebra case. Fix $(Z', g_1,\ldots, g_k)\in f_Z^{-1}(0)$. 

Consider the $i$--th inclusion $G\hookrightarrow O_Z\times G^k$ given by $$g\mapsto (Z', g_1,\ldots, g_{i-1}, g, g_{i+1}, \ldots, g_k).$$
Let $f_Z^i:G\to \Fg^k$ be the composition of this inclusion with $f_Z$. 

The derivative of $f_Z^i$ lies entirely in the $i$--th coordinate of $\Fg^k$, so $$\rank Df_Z\geq \sum_{i=1}^k \rank Df_Z^i.$$ 

If we suppress the unused components of the codomain of $f_Z^i$, we can write $f_Z^i(g)=[Z',\Ad(g)X_i]$. We will compute the rank of $f_Z^i$ at $g_i$. 

We claim that there is an $h\in G$ such that $\Ad(h)Z'=Z$ and $\Ad(hg_i) X_i=\sigma X_i$ for some $\sigma\in \W$. Of course, there is an $h_1\in G$ so that $\Ad(h_1)Z'=Z$. By applying the torus theorem to $\Fn_Z$, we find an $h_2\in G_Z$ so that $\Ad(h_2h_1g_i)X_i \in\Ft$. Set $h=h_2h_1$. At this point we have $\Ad(h)Z'=Z$, $\Ad(hg_i)X_i\in \Ft$. The claim now follows from the fact that if $X$ and $X'$ in $\Ft$ are $\Ad$--related, then they are $\W$--related. To prove the fact: Say $X=\Ad(k_1)X'$. Using the torus theorem, find $k_2\in \Fn_X$ such that $\Ad(k_2k_1)\Ft=\Ft$. Now $\Ad(k_2k_1)X'=X$, and $\Ad(k_2k_1)$ is in the normalizer of the torus, so the action of $\Ad(k_2k_1)$ on $\Ft$ is in the Weyl group. 

Composing $f_Z^i$ with the linear map $\Ad(h)$ does not change its rank. This transformation allows us to assume $Z'=Z$ and $\Ad(g_i)X_i=\sigma X_i$. Now, using the root space decomposition of $\Fg$ we get 
\begin{eqnarray*}
(\im Df_Z^i)_\bC 
&=&
[Z, \Fn_{\sigma X_i}^\perp]_\bC
\\&=&
\left[Z, \bigoplus_{\alpha\in \sigma N_{X_i}}\Fg_\alpha\right]
\\&=&
\bigoplus_{\alpha\in  N_Z\cap \sigma N_{X_i}}\Fg_\alpha.
\end{eqnarray*}
This gives $\rank Df_Z^i= \left| N_Z\cap \sigma N_{X_i} \right|$ and the result follows. 

For the group case, the problem similarly reduces to computing the rank of $g_Z^i(g)=(\Ad(gx_i^{-1}g^{-1})-\id)Z'$, at $g_i$ satisfying $\Ad(g_ix^{-1}g_i^{-1})Z'=Z'$.  There is an $h\in G$ such that $\Ad(h)Z'=Z$ and $h(g_ix_i^{-1}g_i^{-1})h^{-1}=\sigma x$ for some $\sigma\in \W$. Note that $$\Ad(h)g_Z^i(g)=(\Ad( hgx_i^{-1}g^{-1} h^{-1})-\id)\Ad(h)Z'.$$ So as above we may assume $Z'=Z$ and $g_i x_i g_i^{-1}=\sigma x_i$. From here the proof is identical to the algebra case. 
\end{proof}

We may now prove our theorem on when a sum, or product, contains an open set. The proof uses all of the results proven so far. 

\begin{proof}[\emph{\textbf{Proof of Theorem \ref{T:open}}}]
We treat the algebra case, and leave the easy adaptation to the group case to the reader. Assume that the conditions of the theorem hold. For each $Z\in \S$, let $\pi_Z:O_Z\times G^k\to G^k$ be the projection of the domain of $f_Z$ onto $G^k$. 

Again for each $Z\in \S$, the zero set $f_Z^{-1}(0)$ is the union of countably many submanifolds of $O_Z\times G^k$, on each of which $\rank Df_Z$ is constant. This is because $f_Z^{-1}(0)$ is a subvariety, and the set where $f_Z$ has at most any given rank is the union of subvarieties. 

By Lemma \ref{L:rank} and the hypotheses of the theorem, we have $$\rank (f_Z)> \dim O_Z.$$ We conclude that  $f_Z^{-1}(0)$ has co-dimension greater than  $\dim O_Z $; it follows that $\pi_Z f^{-1}_Z(0)$ has measure zero in $G^k$. Hence $$B=\bigcup_{Z\in \S} \pi_Z f^{-1}_Z (0)$$ is measure zero also. For each $(g_1,\ldots, g_k)\notin B$, Lemma \ref{L:maxbad} gives that $\cap_{i=1}^k \Fn_{g_iX_i}=\{0\}$. Using Proposition \ref{P:main}, we conclude that $\sum_{i=1}^k O_{X_i}$ contains an open set. 
\end{proof}

We proceed to some geometric preliminaries followed by our sufficient condition for a sum, or product, not to contain an open set. Given submanifolds $M_1,\ldots, M_k$ of a manifold $M$ we will say that the $M_i$ intersect \emph{totally transversely} at $p\in M$ if $p\in \cap_{i=1}^k M_i$ and, for all proper subsets $S\subset \{1,\ldots, k\}$ and $j\notin S$, we have that $\cap_{i\in S}M_i$ intersects $M_j$ transversely at $p$. This condition is much stronger than pairwise transverse intersections at $p$.

If $M$ is a Riemannian manifold it is equivalent to say that the spaces $(T_p M_i)^\perp$ are linearly independent in $T_pM$. For $k=2$ an intersection is totally transverse at $p\in M_1\cap M_2$ if and only if it is transverse at $p$.  We will need the following general fact. 

\begin{lem}\label{L:tt}
Suppose $M_1,\ldots, M_k$ are submanifolds of a manifold $M$ and that they intersect totally transversely at $p$. Let $g_t^i$ be isotopies of $M_i$, with $g_0^i$ the identity for all $i$. Then for $t$ small enough the intersection $\cap_{i=1}^k g_t^i(M_i)$ is non-empty. 
\end{lem}

\begin{proof}
The $k=2$ case is a standard fact about transversality. The $k>2$ case follows by induction using the fact that $\cap_{i=1}^{k-1} g_t^i(M_i)$ is isotopic to $\cap_{i=1}^{k-1} M_i$ for $t$ small enough and, again for $t$ small enough, $\cap_{i=1}^{k-1} g_t^i(M_i)$ intersects $M_k$ transversely near $p$. 
\end{proof}

\begin{proof}[\emph{\textbf{Proof of Theorem \ref{T:sing}}}]
We consider only the algebra case. Take $Z\in \Ft$ with annihilating root system $\Psi$. Suppose that $N_{X_1}\cap N_{Z}, \cdots, N_{X_k}\cap N_{Z}$ are pairwise disjoint. Define $M_i$ to be the $G_{X_i}$ orbit of $Z$. Recall that $G_{X_i}$ is the stabilizer of $X_i$ and has Lie algebra $\Fn_{X_i}$. 

Note that $M_i$ is a submanifold of $O_Z$ and is contained entirely in $\Fn_{X_i}$. Note further that the tangent space to $Z$ in $M_i$ is $\Fn_{X_i}\cap \Fn_{Z}^\perp$. So the condition on the $X_i$ gives precisely that the $M_i$ intersect totally transversely at $(X_1,\cdots, X_k)$. Hence the previous lemma gives that, for $g_1,\ldots, g_k$ sufficiently close to the identity, $\cap_{i=1}^k M_i$ is non-empty. Since this intersection is contained in $\cap_{i=1}^k \Ad(g_i) n_{X_i}$, we see that this latter intersection is non-empty for all $g_i$ sufficiently close to the identity. Proposition \ref{P:main} now gives that $\sum_{i=1}^k O_{X_i}$ does not contain an open set. 

The proof is identical for the group case. 
\end{proof}


\section{Determination of open sums in $\su(m)$.}\label{S:su_opens}

In this section we will use simpler notation, saying the tuple of root subsystems $(\Phi_{X_1}, \ldots, \Phi_{X_k})$ is open (respectively, singular) if $\sum_{i=1}^k O_{X_i}$ contains an open set (respectively, does not contain an open set). We say ``$(\Phi_1, \ldots, \Phi_k)$ in $A_n$" to indicate that we are studying a tuple $(X_1, \ldots, X_k)$, with all $X_i$ in a Lie algebra $\Fg$ of type $A_n$ (equivalently $\su(m)$, with $m=n+1$), assuming that each $X_i$ has $\Phi_i$ as its annihilating root system.  

\begin{lem}\label{L:su_sing}
The following tuples are singular.
\begin{enumerate}
\item $(A_{r_1},\ldots, A_{r_k})$ in $A_n$ with $\sum_{i=1}^k (n-r_i)\leq n$.
\item $(A_r\times A_r, A_r\times A_r)$ in $A_{2r+1}$.
\end{enumerate}
\end{lem}

\begin{proof}
For the first tuple, set $$\Psi=\{\pm(e_p-e_q):2\leq p<q\leq n+1\} .$$  Pick subsets $S_1,\ldots, S_k$ of $\{1, \ldots, n+1\}$ so that (i) $1\in S_i$ for all $i$, (ii) $S_i$ has size $r_i+1$, and (iii) the sets $S_i^c\cap \{2,\ldots, n+1\}$ are disjoint. Set $$\Phi_i=\{e_p-e_q: p\neq q, p,q\in S_i\}.$$ Note that whenever $X_1,\ldots, X_k\in \Ft$ are of the type in question, we can find $\sigma_i\in \W$ so that $\sigma_i \Phi_{X_i}=\Phi_i$. We see that $$N_i\cap N_\Psi= \{\pm(e_1-e_p): p\in S_i^c\cap \{2,\ldots, n+1\},$$ so the $N_i\cap N_\Psi$ are disjoint. By Theorem \ref{T:sing} the tuple is singular. 

For the second tuple, start with $$\Phi_1=\{\pm(e_p-e_q): 1\leq p\leq r+1, r+2\leq q\leq 2r+2\}.$$ Define
\begin{eqnarray*}
F_o &=&  \{p : 1\leq p\leq r+1, p \text{ odd}\},
\\
F_e &=&  \{p : 1\leq p\leq r+1, p \text{ even}\},
\\
S_o &=& F_o+(r+1),
\\
S_e &=& F_e+(r+1).
\end{eqnarray*}
Set 
\begin{eqnarray*}
\Psi &=& \{\pm(e_p-e_q): p,q\in F_o\cup S_o \text{ or } p, q\in F_e\cup S_e\},
\\
\Phi_2 &=& \{\pm(e_p-e_q): p,q\in F_o\cup S_e \text{ or } p, q\in F_e\cup S_o\}.
\end{eqnarray*}
We note that these are the annihilating root systems of 
$$\begin{array}{cll}
Z =& (1,-1,1, -1, \ldots, &1,-1,1,-1,\ldots ), \\
X_1 =&(1,1, \ldots, 1, & -1, -1, \ldots, -1),\\
X_2 =& (\underbrace{1,-1,1,-1, \ldots}_{r+1},& \underbrace{ -1, 1, -1, 1, \ldots}_{r+1}).
\end{array}$$
Then $N_1\cap N_\Psi$ and $N_2\cap N_\Psi$ are disjoint and again we can apply Theorem \ref{T:sing}.
\end{proof}

It follows that a tuple which is \emph{at least as singular} as one of the above tuples is also singular. Say that $(\Phi_1',\cdots, \Phi_k')$ is at least as singular as $(\Phi_1, \cdots, \Phi_k)$ if $\Phi_i\subset \Phi_i'$ for all $i$. This is the only singularity result we will need, so we turn to the problem of showing a tuple is open. To do so, we apply Theorem \ref{T:open}. This requires the following counting lemma. The author thanks Rolf Hoyer for providing the proof included here. 

\begin{lem}\label{L:eclasses1}
Suppose $\sim$ is an equivalence relation on $\{1,\ldots, m\}$ with all equivalence classes of size at most $w$. Fix $1\leq c \leq m-1$. Then 
$$\left| \{ (p,q): p\sim q, 1\leq p \leq c, c+1\leq q\leq m\} \right| \leq \frac{w}{m} c(m-c)$$
with equality if and only if there are $m/w$ equivalence classes and each contains the same number of elements in $\{1,\ldots, c\}$. 
\end{lem}

\begin{proof}
Suppose there are $d$ equivalence classes, $d\geq \frac{m}{w}$. We induct on $d$. Suppose the $j$--th equivalence class has $a_j$ elements in $\{1,\ldots, c\}$ and $b_j$ elements in $\{c+1,\ldots, m\}$. Without loss of generality we can assume that the $d$--th equivalence class is the largest, and that $a_d+b_d=w$, or else we can replace $w$ with a smaller value and prove this stronger form of the inequality. Set 
$$a=\sum_{j=1}^{d-1} a_j, \quad b=\sum_{j=1}^{d-1} b_j, \quad f=\sum_{j=1}^{d-1} a_j b_j.$$
By induction we may assume $ f\leq \frac{w}{a+b} ab$. The desired inequality is 
$$w(a+a_d)(b+w-a_d)-(w+a+b)(f+a_d(w-a_d))\geq 0.$$
The derivative of this expression with respect to $a_d$ is 
$$-2aw+2a_d(a+b).$$
We find that the only critical point is at $a_d=\frac{wa}{a+b}$ and it is a local minimum. 

It now suffices to check the desired inequality when $f= \frac{w}{a+b} ab$ and $a_d=\frac{wa}{a+b}$. This gives 
\begin{eqnarray*}
&&
w\left( a+\frac{aw}{a+b}\right)\left(b+\frac{bw}{a+b}\right)-\left(w+a+b\right)\left(\frac{wab}{a+b}+\frac{w^2ab}{(a+b)^2} \right)
\\&=&
wab\left( \left(1+\frac{w}{a+b}\right)^2-\left(w+a+b\right)\left(\frac{1}{a+b}+\frac{w}{(a+b)^2}\right)\right)
\\& =&
0.
\end{eqnarray*}

\end{proof}

Now, given $\Phi_X=A_{w_1-1}\times\cdots\times A_{w_d-1}$ in $A_n$, set $w=\max{w_i}$, and $m=n+1$. Consider $\Psi=A_{c-1}\times A_{n-c}$, which is the general form of a co-rank one root subsystem. Define an equivalence relation $\sim$ on $\{1,\ldots, m\}$ by $i\sim j$ if $e_i-e_j\in \Phi_X$. The lemma gives that $$|\Phi_X\cap N_\Psi|\leq \frac{w}{m} |N_\Psi|.$$ Given this result, it will be convenient to rewrite 
$$\sum_{i=1}^k \min_{\sigma\in\W}|N_{i}\cap\sigma N_\Psi|>|N_\Psi|$$ as
$$\sum_{i=1}^k \max_{\sigma\in\W}|\Phi_{i}\cap\sigma N_\Psi|<(k-1)|N_\Psi|.$$

The following result is a restatement of Theorem \ref{T:suopens}. 

\begin{thm}\label{T:su_opens}
Suppose the largest factor of $\Phi_i$ is $A_{r_i}$ for each $i$ . Then $(\Phi_1,\ldots, \Phi_k)$ in $A_n$ is open if and only if $\sum_{i=1}^k (r_i+1)\leq (k-1)(n+1)$ and $(\Phi_1,\ldots, \Phi_k)\neq (A_r\times A_r, A_r\times A_r)$ in $A_{2r+1}$. 
\end{thm}

\begin{proof}
The inequality given is equivalent to $\sum_{i=1}^k (n- r_i)\geq n+1$. If this inequality fails, then $\sum_{i=1}^k (n- r_i)\leq n$, and so the tuple is at least as singular as one of the tuples in Lemma \ref{L:su_sing}, and hence is singular. This lemma also gives that $(A_r\times A_r, A_r\times A_r)$ in $A_{2r+1}$ is singular.

It remains only to show that if $\sum_{i=1}^k (r_i+1)\leq (k-1)(n+1)$ and the the tuple is not $(A_r\times A_r, A_r\times A_r)$ in $A_{2r+1}$ then the tuple is open. 

Lemma \ref{L:eclasses1} gives that
\begin{equation*}\tag{$\diamondsuit$}
\left| \Phi_i\cap N_\Psi\right| \leq \frac{r_i+1}{n+1} \left| N_\Psi \right|
\end{equation*} 
for all $i$. So 
\begin{eqnarray*}
\sum_{i=1}^k \left| \Phi_i\cap N_\Psi \right| 
&\leq& 
\left(\sum_{i=1}^k (r_i+1)\right) \frac{\left| N_\Psi \right|}{n+1}
\\&\leq&
(k-1)(n+1) \frac{\left| N_\Psi \right|}{n+1}= (k-1) \left| N_\Psi \right|.
\end{eqnarray*}
So, it suffices to show that if $\sum_{i=1}^k (r_i+1)=(k-1)(n+1)$ then equality in $(\diamondsuit)$ can hold for all $i$ only if the tuple is $(A_r\times A_r, A_r\times A_r)$ in $A_{2r+1}$. 

Lemma \ref{L:eclasses1} gives that if equality holds for all $i$, then $r_i+1\mid n+1$ for all $i$. Hence $r_i+1\leq \frac{n+1}2$ for all $i$. So 
$$(k-1)(n+1)= \sum_{i=1}^k (r_i+1)\leq k \frac{n+1}2.$$
We conclude that $k=2$ and $r_1+1=r_2+1=\frac{n+1}2$. 

Lemma \ref{L:eclasses1} also gives that the tuple is $(A_r\times A_r, A_r\times A_r)$ in $A_{2r+1}$. This corresponds to the statement that if equality holds, all equivalence classes are of the same size. 
\end{proof}


\section{$L^2$--singular dichotomy in $\su(m)$.}\label{S:su_di}

To prove our $L^2$--singular dichotomy we now need to verify that the conditions of Theorem \ref{T:L2} for the tuples which were determined to be open in the last section. To do so we need the following stronger form of Lemma \ref{L:eclasses1}:

\begin{lem}\label{L:eclasses2}
Let $\sim$ and $\equiv$ be two equivalence relations on $\{1,\ldots, m\}$ so that each $\sim$ equivalence class has size at most $w$ and $\equiv$ has $l$ equivalence classes. Then 
\begin{eqnarray*}
&& 
\left| \{(p,q): 1\leq p<q\leq m, p\sim q, p\not\equiv q\}\right| 
\\&&\leq 
\frac{w}{m}\left| \{(p,q): 1\leq p<q\leq m, p\not\equiv q\}\right|-E
\end{eqnarray*}
where $E=0$ if $\sim$ has $m/w$ equivalence classes and $E=\frac{l}2$ otherwise. 
\end{lem}

\begin{proof}
Let the equivalence classes of $\equiv$ be $T_1,\ldots, T_k$. Define 
\begin{eqnarray*}
N &=& \{(p,q): 1\leq p<q\leq m, p \not\equiv q\}, 
\\
N_s &=& \{(p,q): 1\leq p<q\leq m, \text{ exactly one of $p,q$ is in } T_s\},
\\ 
H &=& \{(p,q)\in N: p\sim q\},
\\
H_s &=& \{(p,q)\in N_s: p\sim q\}.
\end{eqnarray*}
Now, by Lemma \ref{L:eclasses1},
\begin{equation*}\tag{$\clubsuit$}
\left| H_s\right| \leq \frac{w}{m} \left| N_s\right|.
\end{equation*}
Summing this over $s$ and dividing by two we get $\left|H\right| \leq \frac{w}{m}\left| N\right|$, which is what we wished to show when $E=0$. Now, if $\sim$ does not have $m/q$ equivalence classes then $(\clubsuit)$ is never sharp and we get the desired inequality with $E=l/2$. 
\end{proof}

The following theorem completes the proof of our $L^2$--singular dichotomy by establishing the hypotheses for the $L^2$ theorem, Theorem \ref{T:L2}, for every tuple which is open.

\begin{lem}\label{L:L2count}
Suppose $(\Phi_1,\ldots, \Phi_k)$ has $\sum_{i=1}^k (r_i+1)\leq (k-1)(n+1)$ and $(\Phi_1,\ldots, \Phi_k)\neq (A_r\times A_r, A_r\times A_r)$ in $A_{2n+1}$. Let $\Psi$ be a proper $\bR$--closed root subsystem. Then 
$$\sum_{i=1}^k \left| \Phi_i\cap N_\Psi\right|\leq (k-1)\left| N_\Psi\right| -\corank(\Psi).$$
\end{lem}

\begin{proof}
We begin with the claim that 
$$\frac{\left|N_\Psi\right|}{n+1}> \corank(\Psi)$$
for proper $\bR$--closed root subsystems $\Psi$. Among co-rank $d$ subsystems, the left hand size is minimized when $\Psi$ is of type $A_{n-d}$, where the inequality becomes
$$ \frac{\left|N_\Psi\right|}{n+1}=d+\frac{d(n-d)}{n+1}>d.$$
The claim is proved.

Now, say $\sum_{i=1}^k (r_i+1)\leq (k-1)(n+1)-1.$ Using Lemma \ref{L:eclasses2} and the claim, we see that
\begin{eqnarray*}
\sum_{i=1}^k \left| \Phi_i\cap N_\Psi\right| 
&\leq &
\left(\sum_{i=1}^k (r_i+1)\right) \frac{\left|N_\Psi\right|}{n+1}
\\&\leq&
(k-1)\left|N_\Psi\right| - \corank(\Psi)
\end{eqnarray*}
as desired. 

It remains to check the case where $\sum_{i=1}^k (r_i+1)=(k-1)(n+1)-1$. Lemma \ref{L:eclasses2} gives that
$$\left| \Phi_i\cap N_\Psi\right| \leq \frac{r_i+1}{n+1} \left| N_\Psi\right| -E_i$$ where $E_i=\corank(\Psi)+1$ unless $\Phi_i$ is $(A_{r_i})^\frac{n+1}{r_i+1}$, in which case $E_i=0$. (It may seem like the $E$ term has doubled from the lemma, but that is because the lemma only considers positive roots.)

We have to prove that we cannot have all $E_i=0$. If this were the case, we would have $r_i+1\mid n+1$, so $r_i+1\leq \frac{n+1}2$ for all $i$. As in Theorem \ref{T:su_opens} we get that $k=2$ and the tuple is $(A_r\times A_r, A_r\times A_r)$ in $A_{2r+1}$.
\end{proof}


\bibliography{mybib}{}
\bibliographystyle{amsplain}
\end{document}